\documentclass[12pt]{amsart}
\usepackage{amsmath}
\usepackage{amsxtra}
\usepackage{enumitem}
\usepackage{amscd}
\usepackage{tikz-cd}
\usepackage{amsthm}
\usepackage{hyperref}
\usepackage[T1]{fontenc}
\usepackage{amsfonts}
\usepackage{amssymb}
\usepackage{eucal}
\usepackage[all]{xypic}
\usepackage{color} 
\newtheorem{Theo}[subsubsection]{Theorem}
\newtheorem{Theor}{Theorem}
\newtheorem{Theore}{Theorem}

\newtheorem{cor}[subsubsection]{Corollary}

\theoremstyle{definition}

\newtheorem{fact}[subsubsection]{Fact} 
\newtheorem{rem}[subsubsection]{Remark}

\newtheorem{question}[subsection]{Question}
\newtheorem{Prop}[subsubsection]{Proposition}
\newtheorem{Obs}[subsubsection]{Observation}
\newtheorem{example}[subsubsection]{Example}
\newtheorem{Lemm}[subsubsection]{Lemma} 

\newtheorem{definition}[subsubsection]{Definition}
\newtheorem{problem}[subsubsection]{Problem}

\newtheorem{ex}[subsection]{Exercise}

\usepackage{yhmath}
\DeclareSymbolFont{largesymbols}{OMX}{yhex}{m}{n}
\DeclareMathAccent{\widetilde}{\mathord}{largesymbols}{"65}

\newcommand{\bp}{\begin{Prop}}
\newcommand{\ep}{\end{Prop}}

\newcommand{\bl}{\begin{Lemm}}
\newcommand{\el}{\end{Lemm}}
\newcommand{\bex}{\begin{ex} \rm}
\newcommand{\eex}{\end{ex}}
\newcommand{\bt}{\begin{Theo}}
\newcommand{\et}{\end{Theo}}
\newcommand{\bq}{\begin{question}}
\newcommand{\eq}{\end{question}}

\newcommand{\bc}{\begin{cor}}
\newcommand{\ec}{\end{cor}}
\newcommand{\bob}{\begin{Obs}}
\newcommand{\eob}{\end{Obs}}

\newcommand{\nc}{\newcommand}
\nc{\renc}{\renewcommand}
\nc{\ssec}{\subsection}
\nc{\sssec}{\subsubsection} 
\nc\ol{\overline}
\nc\wt{\widetilde}
\nc\wh{\widehat}
\nc\tboxtimes{\wt{\boxtimes}}

\renc{\d}{{\delta}}
\nc{\Aa}{{\mathbb{A}}}
\nc{\Bb}{{\mathbb{B}}}
 \nc{\Gg}{{\mathbb{G}}}  
\nc{\Hh}{{\mathbb{H}}}
 \nc{\Nn}{{\mathbb{N}}}
\nc{\Pp}{{\mathbb{P}}}
\nc{\Rr}{{\mathbb{R}}}
\newcommand{\F}{\mathbb{F}}
\nc{\BV}{{\mathbb{V}}}
\nc{\BW}{{\mathbb{W}}}
\newcommand{\Z}{\mathbb{Z}}
\newcommand{\N}{\mathbb{N}}

\nc{\Qq}{{\mathbb{Q}}}
\nc{\Ss}{{\mathbb{S}}}
\nc{\Cc}{{\mathbb{C}}}
\nc{\Ff}{{\mathbb{F}}}
 \nc{\EL}{{L_\infty}}

\nc{\CA}{{\mathcal{A}}}
\nc{\CB}{{\mathcal{B}}}

\nc{\CE}{{\mathcal{E}}}
\nc{\CF}{{\mathcal{F}}}

\nc{\Las}{\mathsf{Las}}
\nc{\CG}{{\mathcal{G}}}

\nc{\CL}{{\mathcal{L}}}
\nc{\CC}{{\mathcal{C}}}
\nc{\CM}{{\mathcal{M}}}

\nc{\CN}{{\mathcal{N}}}
\nc{\Oog}{{\mathbb{O}}}
\nc{\Oo}{{\mathcal{O}}}
\nc{\CP}{{\mathcal{P}}}
\nc{\CQ}{{\mathcal{Q}}}
\nc{\CR}{{\mathcal{R}}}
\nc{\CS}{{\mathcal{S}}}
\nc{\CT}{{\mathcal{T}}}
\nc{\CU}{{\mathcal{P}}}

\nc{\CV}{{\mathcal{V}}}
 
\nc{\CW}{{\mathcal{W}}}
\nc{\CZ}{{\mathcal{Z}}}

\nc{\cM}{{\check{\mathcal M}}{}}
\nc{\csM}{{\check{\mathcal A}}{}}
\nc{\oM}{{\overset{\circ}{\mathcal M}}{}}
\nc{\obM}{{\overset{\circ}{\mathbf M}}{}}
\nc{\oCA}{{\overset{\circ}{\mathcal A}}{}}
\nc{\obA}{{\overset{\circ}{\mathbf A}}{}}
\nc{\ooM}{{\overset{\circ}{M}}{}}
\nc{\osM}{{\overset{\circ}{\mathsf M}}{}}
\nc{\vM}{{\overset{\bullet}{\mathcal M}}{}}
\nc{\nM}{{\underset{\bullet}{\mathcal M}}{}}
\nc{\oD}{{\overset{\circ}{\mathcal D}}{}}
\nc{\obD}{{\overset{\circ}{\mathbf D}}{}}
\nc{\oA}{{\overset{\circ}{\mathbb A}}{}}
\nc{\op}{{\overset{\bullet}{\mathbf p}}{}}
\nc{\cp}{{\overset{\circ}{\mathbf p}}{}}
\nc{\oU}{{\overset{\bullet}{\mathcal U}}{}}
\nc{\oZ}{{\overset{\circ}{\mathcal Z}}{}}
\nc{\ofZ}{{\overset{\circ}{\mathfrak Z}}{}}
\nc{\oF}{{\overset{\circ}{\fF}}}

\nc{\fa}{{\mathfrak{a}}}
\nc{\fb}{{\mathfrak{b}}}
\nc{\fg}{{\mathfrak{g}}}
\nc{\fgt}{{\fg}_!}
\nc{\fgl}{{\mathfrak{gl}}}
\nc{\fh}{{\mathfrak{h}}}
\nc{\fj}{{\mathfrak{j}}}
\nc{\fm}{{\mathfrak{m}}}
\nc{\ft}{{\mathfrak{t}}}
\nc{\fn}{{\mathfrak{n}}}
\nc{\fu}{{\mathfrak{u}}}
\nc{\fp}{{\mathfrak{p}}}
\nc{\fr}{{\mathfrak{r}}}
\nc{\fs}{{\mathfrak{s}}}
\nc{\fsl}{{\mathfrak{sl}}}
\nc{\hsl}{{\widehat{\mathfrak{sl}}}}
\nc{\hgl}{{\widehat{\mathfrak{gl}}}}
\nc{\hg}{{\widehat{\mathfrak{g}}}}
\nc{\chg}{{\widehat{\mathfrak{g}}}{}^\vee}
\nc{\hn}{{\widehat{\mathfrak{n}}}}
\nc{\chn}{{\widehat{\mathfrak{n}}}{}^\vee}

\nc{\fA}{{\mathfrak{A}}}
\nc{\ACVF}{{\text{ACVF}}}
\nc{\fB}{{\mathfrak{B}}}
\nc{\fD}{{\mathfrak{D}}}
\nc{\fE}{{\mathfrak{E}}}
\nc{\fF}{{\mathfrak{F}}}
\nc{\fG}{{\mathfrak{G}}}
\nc{\fK}{{\mathfrak{K}}}
\nc{\fL}{{\mathfrak{L}}}
\nc{\fM}{{\mathfrak{M}}}
\nc{\fN}{{\mathfrak{N}}}
\nc{\fP}{{\mathfrak{P}}}
\nc{\fU}{{\mathfrak{U}}}
\nc{\fV}{{\mathfrak{V}}}
\nc{\fZ}{{\mathfrak{Z}}}
\newcommand{\Q}{\mathbb{Q}}
\nc{\bb}{{\mathbf{b}}}
\nc{\bd}{\partial}
\nc{\be}{{\mathbf{e}}}
\nc{\bj}{{\mathbf{j}}}
\nc{\bn}{{\mathbf{n}}}
\nc{\bF}{{\mathbf{F}}}
\nc{\bu}{{\mathbf{u}}}
\nc{\bv}{{\mathbf{v}}}
\nc{\bx}{{\mathbf{x}}}
\nc{\bs}{{\mathbf{s}}}
\nc{\by}{{\bar{y}}}
\nc{\bw}{{\mathbf{w}}}
\nc{\bA}{{\mathbf{A}}}
\nc{\bK}{{\mathbf{K}}}
\nc{\bI}{{\mathbf{I}}}
\nc{\bB}{{\mathbf{B}}}
\nc{\bG}{{\mathbf{G}}}

\nc{\bD}{{\mathbf{D}}}
\nc{\bP}{{\mathbf{P}}}
\nc{\bH}{{\mathbf{H}}}
\nc{\bM}{{\mathbf{M}}}
\nc{\bN}{{\mathbf{N}}}
\nc{\bV}{{\mathbf{V}}}
\nc{\bU}{{\mathbf{U}}}
\nc{\bL}{{\mathbf{L}}}

\nc{\bW}{{\mathbf{W}}}
\nc{\bX}{{\mathbf{X}}}
\nc{\bY}{{\mathbf{Y}}}
\nc{\bZ}{{\mathbf{Z}}}
\nc{\bS}{{\mathbf{S}}}
\nc{\bSi}{{\bar{\Sigma}}}
\nc{\sA}{{\mathsf{A}}}
\nc{\sB}{{\mathsf{B}}}
\nc{\sC}{{\mathsf{C}}}
\nc{\sD}{{\mathsf{D}}}
\nc{\sF}{{\mathsf{F}}}
\nc{\sG}{{\mathsf{G}}}
\nc{\sK}{{\mathsf{K}}}
\nc{\sM}{{\mathsf{M}}}
\nc{\sO}{{\mathsf{O}}}
\nc{\sQ}{{\mathsf{Q}}}
\nc{\sP}{{\mathsf{P}}}
\nc{\sZ}{{\mathsf{Z}}}
\nc{\sfp}{{\mathsf{p}}}
\nc{\sr}{{\mathsf{r}}}
\nc{\sg}{{\mathsf{g}}}
\nc{\sff}{{\mathsf{f}}}
\nc{\sfb}{{\mathsf{b}}}
\nc{\sfc}{{\mathsf{c}}}
\nc{\sd}{{\ltimes}} 

\nc{\tH}{{\widetilde{H}}}
\nc{\tA}{{\widetilde{\mathbf{A}}}}
\nc{\tB}{{\widetilde{\mathcal{B}}}}
\nc{\tg}{{\widetilde{\mathfrak{g}}}}
\nc{\tG}{{\widetilde{G}}}

\nc{\TM}{{\widetilde{\mathbb{M}}}{}}
\nc{\tO}{{\widetilde{\mathsf{O}}}{}}
\nc{\tU}{\widetilde{U}}
\nc{\TZ}{{\tilde{Z}}}
\nc{\tx}{{\tilde{x}}}
\nc{\tq}{{\tilde{q}}}
\nc{\Trop}{\text{Trop}}
\nc{\tfP}{{\widetilde{\mathfrak{P}}}{}}
\nc{\tz}{{\tilde{\zeta}}}
\nc{\tmu}{{\tilde{\mu}}}

  \nc{\vol}{{\mathop{\operatorname{\rm vol\,}}}}
  
  \nc{\gal}{{\mathop{\operatorname{\rm Gal\,}}}}
  \nc{\cl}{{\mathop{\operatorname{\rm cl}}}}
  \nc{\disc}{{\mathop{\operatorname{\rm disc}}}}
  \nc{\Sym}{{\mathop{\operatorname{\rm Sym}}}}
   \nc{\Aut}{{\mathop{\operatorname{\rm Aut}}}}
 \nc{\Spec}{{\mathop{\operatorname{\rm Spec}}}}
  \nc{\spec}{{\mathop{\operatorname{\rm Spec}}}}
\nc{\Ker}{{\mathop{\operatorname{\rm Ker}}}}
 \nc{\dom}{{\mathop{\operatorname{\rm dom}}}}
\nc{\End}{{\mathop{\operatorname{\rm End}}}}
 \nc{\Hom}{\operatorname{\Hom}}
 \nc{\GL}{{\mathop{\operatorname{\rm GL}}}}
 \nc{\Id}{{\mathop{\operatorname{\rm Id}}}}
 \nc{\rk}{{\mathop{\operatorname{\rm rk}}}}
 \nc{\length}{{\mathop{\operatorname{\rm length}}}}
\nc{\supp}{{\mathop{\operatorname{\rm supp} \, }}}
\nc{\val}{{\rm val}}

\def\Ind#1#2#3{{#1} {\downarrow}_{#3} {#2} }

\nc{\seq}[1]{\stackrel{#1}{\sim}}

\def\beq#1{\begin{equation} \label{ #1}}
\def\eeq{\end{equation}}

\def\prf{\begin{proof}}
\def\pv{\end{proof} }
 \def\eprf{\end{proof} }

 \renc{\b}{{\beta}}

\def\Ind#1#2{#1\setbox0=\hbox{$#1x$}\kern\wd0\hbox to 0pt{\hss$#1\mid$\hss}
\lower.9\ht0\hbox to 0pt{\hss$#1\smile$\hss}\kern\wd0}

  \def\C{\mathbb{C}}

\usepackage{color}
\usepackage{hyperref}

\usepackage{url}

 \makeatletter
\def\@setthanks{\vspace{-\baselineskip}\def\thanks##1{\@par##1\@addpunct.}\thankses}
\makeatother
\title{Valued FIELDs WITH A total RESIDUE MAP}
\author{Konstantinos Kartas}
\thanks{During this research, the author was funded by EPSRC grant EP/20998761 and was also supported by the Onassis Foundation - Scholarship ID: F ZP 020-1/2019-2020.}

\newcommand{\Addresses}{{
  \bigskip
  \footnotesize

\textsc{Mathematical Institute, Woodstock Road, Oxford OX2 6GG.}\par\nopagebreak
  \textit{E-mail address}: \texttt{kartas@maths.ox.ac.uk}
}}

\begin{document}
\maketitle

\begin{abstract}
When $k$ is a \textit{finite} field, Becker-Denef-Lipschitz (1979) observed that the \textit{total residue} map $\text{res}:k(\!(t)\!)\to k$, which picks out the constant term of the Laurent series, is definable in the language of rings with a parameter for $t$. Driven by this observation, we study the theory $\text{VF}_{\text{res},\iota}$ of valued fields equipped with a linear form $\text{res}:K\to k$ which restricts to the residue map on the valuation ring. We prove that $\text{VF}_{\text{res},\iota}$ does not admit a model companion. In addition, we show that $(k(\!(t)\!),\text{res})$ is \textit{undecidable} whenever $k$ is an \textit{infinite} field. As a consequence, we get that $(\Cc(\!(t)\!), \text{Res}_0)$ is undecidable, where $\text{Res}_0:f\mapsto \text{Res}_0(f)$ maps $f$ to its complex residue at $0$.  
\end{abstract}
\setcounter{tocdepth}{1}
\section*{Introduction}
Let $k$ be a field. Consider the power series field $k(\!(t)\!)$ and let $\text{res}:k(\!(t)\!)\to k$ be the \textit{total residue} map 
$$\text{res}:\sum_{i=-N}^{\infty} c_i t^i \mapsto c_0$$ 
Viewing $k(\!(t)\!)$ as a $k$-vector space, we see that $\text{res}:k(\!(t)\!)\to k$ is a linear form which extends the usual residue map $\pi :k[\![t]\!]\to k$ to all of $k(\!(t)\!)$ (hence the name \textit{total}). The motivation for studying the model-theory of this structure is twofold:
\begin{enumerate}
\item In complex analysis, one defines the \textit{residue} of a complex meromorphic function $f\in \C(\!(t)\!)$ at an isolated singularity $a\in \C$, denoted by $\text{Res}(f,a)$ or $\text{Res}_a(f)$. The map $\text{res}$ is essentially a shifted version of $\text{Res}_0:\C(\!(t)\!)\to \C: f\mapsto \text{Res}_0(f)$, namely $\text{res}(t\cdot f)=\text{Res}_0(f)$. 
\item 
Becker-Denef-Lipschitz \cite{BDL} showed for $x\in \F_p(\!(t)\!)$ that $\text{res}(x)=0$ precisely when there exist $x_0,x_1,...,x_{p-1}\in \F_p(\!(t)\!)$ such that
$$x=x_0^p-x_0+tx_1^p+....+t^{p-1}x_{p-1}^p$$ 
This easily implies that $\text{res}:\F_p(\!(t)\!)\to \F_p$ is definable in $L_{\text{rings}}$ with a parameter for $t$. Becker-Denef-Lipschitz used this to show that $\F_p(\!(t)\!)$ is undecidable in the language of valued fields with a cross-section. This was in sharp contrast with the result by Ax-Kochen \cite{AK3} and Ershov \cite{Ershov}: For any decidable field $k$ of characteristic $0$, the power series field $k(\!(t)\!)$ is decidable in the language of valued fields with a cross-section.
\end{enumerate}
Towards understanding the model theory of $\F_p(\!(t)\!)$, it may be instructive to isolate such definable functions (or predicates) and study them over $k(\!(t)\!)$, where $k$ is not necessarily finite. 
This approach is largely influenced by Cherlin \cite{Cherlin}, especially Problems 3 and 4 in \S 5 \cite{Cherlin}. 

In the present paper, we study the model theory of $k(\!(t)\!)$ equipped with $\text{res}:k(\!(t)\!)\to k$. We also take this a step further and study valued fields---not necessarily power series fields---enriched with a total residue map $\text{res}:K\to k$. In the axiomatic setting, the \textit{total residue} map $\text{res}:K\to k$ is assumed to be a linear form extending the usual residue map $\pi:\Oo_K \to k$, where $\Oo_K$ is the valuation ring of $K$. We consider the theory of equal characteristic valued fields $K$ with a total residue map $\text{res}:K\to k$ and a \textit{lift} $\iota:k\to K$ of the residue field, namely $\iota$ is a field embedding such that $\pi \circ \iota=\text{id}_k$. We call $\text{VF}_{\text{res},\iota}$ the resulting theory. 

At first glance, the theory $\text{VF}_{\text{res},\iota}$ seems like an innocent variant of $\text{VF}$. However, we will show the following:
\begin{Theore} \label{modelcomp}
The theory $\text{VF}_{\text{res},\iota}$ does \textit{not} admit a model companion.
\end{Theore}  
This should be contrasted with the fact (due to A. Robinson) that the theory VF of valued fields admits a model companion, namely ACVF, the theory of algebraically closed valued fields. The theory $\text{VF}_{\iota}$ of valued fields with a lift of the residue field still admits a model companion, namely the theory $\text{ACVF}_{\iota}$ described by Hrushovski-Kazhdan \S 6 \cite{HK} (there it is called $T_{loc}$). Indeed, they prove a quantifier-elimination result for $\text{ACVF}_{\iota}$, which in fact---according to \S 6.1 \cite{HK}---goes back to F. Delon. 

We also prove: 
\begin{Theore}
Let $k$ be an infinite field. Then $k(\!(t)\!)$ is undecidable in $L_{\text{res}}$.
\end{Theore} 
Here $L_{\text{res}}$ is the three-sorted language $L_{\text{val}}$  of valued fields together with a function symbol for $\text{res}:k(\!(t)\!)\to k$. In fact, we prove that the $\exists \forall$-theory is undecidable---at least when $t$ is added in the language---and also that $(\N,+,\cdot)$ is interpretable (without parameters). 
The proof also applies to the Hahn series field $k(\!(t^{\Gamma})\!)$, where $\Gamma$ is any non-trivial ordered abelian group, and also to the Puiseux series field $k\{\!\{t\}\!\}=\bigcup_{n\in \N} k(\!(t^{1/n})\!)$. As an application, we get that $\Cc(\!(t)\!)$ and $\Cc\{\!\{t\}\!\}$ are undecidable in the language of valued fields together with a function symbol for $\text{Res}_0:f\mapsto \text{Res}_0(f)$ which maps $f$ to its complex residue at $0$ (see Corollary \ref{complexnumbers}). 
\section{Valued Fields with a total residue map}
We study valued fields of equal characteristic, together with a lift $\iota:k\to K$ and a \textit{total} residue map $\text{res}:K\to k$, which is a linear form extending the residue map $\pi:\Oo_K\to k$. To simplify notation, we identify $x\in k$ with its image $\iota(x)$ in $K$. It will always be clear from context where such an $x$ lives.
\subsection{Axiomatization of $\text{VF}_{\text{res},\iota}$}
Let $L_{\text{val}}$ be the three-sorted language of valued fields with sorts for the field, the value group and the residue field and a function symbol for $v:K\to \Gamma\cup \{\infty\}$. We call $L_{\text{res},\iota}$ the enrichment of $L_{\text{val}}$ which includes function symbols for $\iota$ and $\text{res}$. Consider the following set of axioms in $L_{\text{res},\iota}$:
\begin{enumerate}
\item $(K,v)$ is a valued field of equal characteristic and $\iota:k\to K$ is a field embedding such that $\pi \circ \iota= \text{id}_k$.
\item We have that $\text{res}|_{\Oo_K}=\pi$.
\item $\text{res}: K\to k$ is $k$-linear, i.e. $\text{res}(\lambda a+\mu b)=\lambda \text{res}(a)+\mu r(b)$, for all $\lambda,\mu\in k$.
\end{enumerate} 
Let $\text{VF}_{\text{res},\iota}$ be the $L_{\text{res},\iota}$-theory generated by the above axioms. 
\begin{example} \label{basic}
Let $k$ be a field.
\begin{enumerate}[label=(\roman*)]

\item Let $\Gamma$ be an ordered abelian group and $k(\!(t^{\Gamma})\!)$ be the Hahn series field over $k$ with value group $\Gamma$. We have $(k(\!(t^{\Gamma})\!),v_t,\text{res},\iota)\models \text{VF}_{\text{res},\iota}$, where 
$$\text{res}:\sum_{q\in \Gamma } c_q t^q \mapsto c_0$$ 
and $\iota:k\to k(\!(t^{\Gamma})\!)$ is the obvious lift. 
\item  Similarly, $(k\{\!\{t\}\!\},v_t,\text{res},\iota)\models \text{VF}_{\text{res},\iota}$, where $k\{\!\{t\}\!\}=\bigcup_{n\in \N} k(\!(t^{1/n})\!)$ is the Puiseux series field over $k$.
\end{enumerate}

\end{example}

\subsection{Extensions of $L_{\text{res},\iota}$-structures}

\bl \label{minlinindep}
Let $(K,v,\iota)\subseteq (K',v',\iota')$ be an extension of valued fields with lifts of their residue fields. If $\beta_1,...,\beta_n\in k'$ are $k$-linearly independent and $a_1,...,a_n\in K$, then 
$$v'( \sum_{i=1}^n \beta_i a_i)=\min_{1\leq i\leq n} v(a_i)$$
\el 
\begin{proof}
We may assume that $v(a_1)=...=v(a_n)$. For each $1\leq i\leq n$, there exist $\alpha_i\in k$ and $\varepsilon_i \in \mathfrak{m}$, such that $a_i=\alpha_i \cdot a_1+\varepsilon_i\cdot a_1$. It follows that 
$$ \sum_{i=1}^n \beta_i a_i= a_1\cdot (\sum_{i=1}^n \alpha_i \beta_i  + \sum_{i=1}^n \beta_i\varepsilon_i) $$
Since the $\beta_i$'s are $k$-linearly independent, we get that $\sum_{i=1}^n \alpha_i \beta_i \neq 0$ and hence $v'(\sum_{i=1}^n \alpha_i \beta_i)=0$. Since $v'(\sum_{i=1}^n \beta_i\varepsilon_i)>0$, we get that 
$$v'(\beta_1a_1+...+\beta_n a_n)=v(a_1)+v(\sum_{i=1}^n \alpha_i \beta_i  + \sum_{i=1}^n \beta_i\varepsilon_i)=v(a_1) $$
as needed.
\end{proof}
Given $(K,v,\iota)\subseteq (K',v',\iota')$, we denote by $\langle K \rangle_{k'}$ the $k'$-linear span of $K$ inside $K'$. Note that $\langle K \rangle_{k'}$ is isomorphic to $K\otimes_k k'$ as an $k'$-vector space.
\bl \label{obviousintersection}
Let $(K,v,\iota)\subseteq (K',v',\iota')$ be an extension of valued fields with lifts of their residue fields. Then $\Oo_{K'}\cap \langle K \rangle_{k'} =\langle \Oo_K \rangle_{k'} $. 
\el 
\begin{proof}
First, we prove the following:\\
\textbf{Claim: } Let $a_1,...,a_n\in K$ be $k$-linearly independent over $\Oo_K$. Then $a_1,...,a_n$ are also $l$-linearly independent over $\Oo_{K'}$. 
\begin{proof}
Suppose $\beta_1',...,\beta_n'\in k'$ are such that  
$$\sum_{i=1}^n \beta_i' a_i\in \Oo_{K'}$$
Let $\{\beta_1,...,\beta_m\}$ be a $k$-linear basis of $\langle \beta_1',...,\beta_n'\rangle_k$ and write $\beta_i'= \sum_{j=1}^m c_{ij} \cdot \beta_i$ with $c_{ij}\in k$. We will then have that
$$\sum_{i=1}^n \beta_i' a_i = \sum_{i=1}^n (\beta_i \sum_{j=1}^m c_{ij} a_j) $$
Since the $\beta_i$'s are $k$-linearly independent, Lemma \ref{minlinindep} implies that
$$v'(\sum_{i=1}^n (\beta_i \sum_{j=1}^m c_{ij} a_j) ))=\min_{1\leq i\leq n} \{ v( \sum_{j=1}^m c_{ij} a_j)\}$$
For $i=1,...,n$, it follows that
$$\sum_{j=1}^m c_{ij} a_j\in \Oo_K$$
Since the $a_i$'s are $k$-linearly independent over $\Oo_K$, we get that $ c_{ij}=0$ for all $i,j$. Therefore $\beta_i'=0$ for all $i$ and the $a_i$'s are $k'$-linearly independent over $\Oo_{K'}$. 
\qedhere $_{\textit{Claim}}$ \end{proof}
Now let $V$ be a complement of the $k$-vector subspace $\Oo_K\subseteq K$, i.e., we have $K=\Oo_K\oplus V$. We will then have that $\langle K \rangle_{k'}= \langle \Oo_K \rangle_{k'} + V_{k'}$. By the claim, we get that $ V_{k'} \cap \Oo_{K'}=\{0\}$ and since $\langle \Oo_K \rangle_{k'}\subseteq \Oo_{K'}$, we conclude that $\Oo_{K'}\cap \langle K \rangle_{k'} =\langle \Oo_K \rangle_{k'} $.
\end{proof}

\bl \label{elementarylinearalgebra2}
Let $(K,v,\text{res},\iota)\models \text{VF}_{\text{res},\iota}$ and $(K,v,\iota)\subseteq (K',v',\iota')$. Let $\{e_i:i\in I\}\subseteq L$ be  $k'$-linearly independent over $\Oo_{K'}+\langle K \rangle_{k'}$. Then there exists an $k'$-linear map $\text{res}':K'\to k'$ extending $\pi_{K'}: \Oo_{K'}\to k'$ with  $\text{res}'(e_i)=0$ for all $i\in I$ and such that $(K',v',\text{res}',\iota')$ is a model of $\text{VF}_{\text{res},\iota}$ extending $(K,v,\text{res},\iota)$.
\el 
\begin{proof}
First we extend $\text{res}:K\to k$ to $\text{res}':\langle K \rangle_{k'} \to k'$ by extension of scalars:
$$\text{res}'(\beta\cdot a)= \beta \cdot \text{res}(a)$$ 
for $\beta \in k'$ and $a\in K$. 
By Lemma \ref{obviousintersection}, if $b\in \Oo_{K'}\cap \langle K \rangle_{k'} $, we may write $b=\sum_{i=1}^n \beta_i a_i$, with $\beta_i\in k'$ and $a_i\in \Oo_K$. We now compute that 
$$\text{res}'(b)=\sum_{i=1}^n \beta_i\cdot \text{res}'(a_i)= \sum_{i=1}^n \beta_i\cdot \text{res}(a_i)=\sum_{i=1}^n \beta_i\cdot \pi_K(a_i)=\pi_L(b)$$
We may therefore extend $\text{res}:K\to k$ to $\text{res}':\Oo_{K'}+\langle K \rangle_{k'} \to k'$ so that it restricts to $\text{res}$ on $K$ and also to $\pi_{K'}$ on $\Oo_{K'}$. Finally, we extend the linear map $\text{res}'$ to $K'$ by requiring that $\text{res}(e_i)= 0$ and get a model $(K',v',\text{res}',\iota')$ of $\text{VF}_{\text{res},\iota}$ extending $(K,v,\text{res},\iota)$, as required.
\end{proof}
\subsection{Structures on the rational function field}

\subsubsection{Extension by an infinitesimal}

\begin{fact} [Corollary 2.2.3 \cite{engprest}] \label{engprest1}
Let $(K,v)$ be a valued field with value group $\Gamma$ and residue field $k$. Let $\Gamma'$ be an ordered abelian group extending $\Gamma$ and $\gamma\in \Gamma'$ be torsion-free over $\Gamma$, i.e., if $n\cdot \gamma \in \Gamma$, then $n=0$. Then there is exactly one valuation $w$ on $K(X)$ extending $v$ with $w(X)=\gamma$. We have $\Gamma_{K(X)}=\Gamma \oplus \Z \gamma$, with the ordering induced by $\Gamma'$, and $k_{K(X)}=k$. 
\end{fact} 

\begin{rem}\label{lift}
\begin{enumerate}[label=(\roman*)]
\item It follows that there is a unique valuation $w$ on $K(X)$ such that $wx>\Gamma$. We will have $k_{K(X)}=k$ and $\Gamma_{K(X)}= \Z\gamma \oplus_{lex}\Gamma$.

\item A lift $\iota:k \to K$ of the valued field $(K,v)$ naturally induces a lift of $(K(X),w)$, namely $\iota_{K(X)}:k\to K(X):\alpha \mapsto \iota(\alpha)$.

\end{enumerate}
\end{rem}

\bl \label{funcfieldstruct}
Let $(K,v,\text{res},\iota)\models \text{VF}_{\text{res},\iota}$. Let $w$ be the unique valuation on $K(X)$ such that $w(X)>\Gamma$ and also let $\iota_{K(X)}:k\to K(X)$ be the induced lift. Then there is a map $\text{res}_{K(X)}:K(X)\to k$ with $\text{res}_{K(X)}(X^{-m})=0$ for all $m\in \N$, such that $(K(X);w,\text{res}_{K(X)},\iota_{K(X)})$ is a model of $\text{VF}_{\text{res},\iota}$ extending $(K,v,\text{res},\iota)$.
\el 
\begin{proof}
By Lemma \ref{elementarylinearalgebra2}, it suffices to show that the elements $X^{-1},X^{-2},...,X^{-n}$ are $k$-linearly independent over $K+\Oo_{K(X)}$. Indeed, for any $c \in K$ and $\beta_1,...,\beta_n\in k$ with $\beta_n\neq 0$, we get that 
$$w(\sum_{i=1}^n \beta_i X^{-i}+c )=-n\cdot w(X)<0$$
using that $w(X)>\Gamma$.
\end{proof}

\bc \label{lemmapowers}
Let $(K,v,\text{res},\iota)$ be an e.c. model of $\text{VF}_{\text{res},\iota}$ and $n\in \N$. Then there exists $a\in \mathfrak{m} - \{0\}$ such that $\text{res}(a^{-m})=0$ for $m=1,...,n$.
\ec
\begin{proof}
Immediate from Lemma \ref{funcfieldstruct}.
\end{proof}

\subsubsection{Gauss valuation}
The valuation described below is known as the \textit{Gauss extension} of $v$ from $K$ to $K(X)$.
\begin{fact} [Corollary 2.2.2 \cite{engprest}] \label{engprest2}
Let $(K,v)$ be a valued field with value group $\Gamma$ and residue field $k$. There exists a unique valuation $w$ on $K(X)$ extending $v$  such that $w(X)=0$ and the residue $x$ of $X$ is transcendental over $k$. This valuation is defined by the formula
$$w(a_0+a_1X+....+a_nX^n)=\min_{1\leq i\leq n} v(a_i)$$
for $a_i\in K$. We have $k_{K(X)}=k(x)$ and $\Gamma_{K(X)}=\Gamma$.
\end{fact} 
\begin{rem} \label{descriptionofgauss}
A lift $\iota:k \to K$ automatically induces a lift
$$\iota_{K(X)}:k(x)\to K(X):f(x)\mapsto \iota(f)(X)$$ 
Namely, $\iota_{K(X)}$ agrees with $\iota$ on $K$ and maps $x$ to $X$. Thus, the image of $\iota_{K(X)}$ equals $k(X)$, once again identifying $k$ with its image $\iota(k) \subseteq K$ via $\iota$.
\end{rem}

\bl \label{trickylemma}
Let $(K,v,\text{res},\iota)\models  \text{VF}_{\text{res},\iota}$ and $a,c\in K$ with $a\in \mathfrak{m}-\{0\}$ and $v(c)<\Z v(a)$. Let $w$ be the Gauss extension of $v$ to $K(X)$ and $\iota_{K(X)}:k(x)\to K(X)$ be the induced lift. There is a map $\text{res}_{K(X)}:K(X)\to k(x)$ with $\text{res}_{K(X)}(\frac{c}{1-a X })=0$, such that $(K(X);w,\text{res}_{K(X)},\iota_{K(X)})$ is a model of $\text{VF}_{\text{res},\iota}$ extending $(K,v,\text{res},\iota)$. 
\el 
\begin{proof}
By Lemma \ref{elementarylinearalgebra2}, it suffices to show that $\frac{c}{1-aX}$ is $k(X)$-linearly independent over $\Oo_{K(X)}+\langle K \rangle _{k(X)}$. Note that $\langle K \rangle _{k(X)}= K[X]_{k[X]\backslash \{0\}}$, the latter being the localization of $K[X]$ at $k[X]\backslash \{0\}\subseteq K[X]$. 

Suppose for a contradiction that 
$$\frac{c}{1-aX} + \frac{f(X)}{g(X)}\in \Oo_{K(X)}$$ for some $f(X)\in K[X]$ and $g(X)\in k[X]\backslash \{0\}$. Moreover, choose $f(X)$ and $g(X)$ as above such that $\deg(f)$ is minimum. Note that $f(X)\neq 0$ because $w(\frac{c}{1-aX} )<0$ and hence $\deg(f)$ is well-defined. Write $f(X)=f_0+...+f_nX^n$ and $g(X)=g_0+....+g_nX^n$ with  $f_i\in K$ and $g_i\in k$. Note that 
$$w((1-aX)g(X))=w(1-aX)+w(g(X))=0$$ 
It follows that $w(c g(X)+(1-aX)f(X))\geq 0$. By the definition of $w$, this means that
$$(1)\  v(f_0+cg_0)\geq 0, \ (2)\  v(f_m+cg_m-af_{m-1})\geq 0 \mbox{ and } (3)\ v(af_n)\geq 0 $$ 
for $m=1,...,n$.\\
\textbf{Claim 1: }We have that $v(f_0)<0$.
\begin{proof}
Suppose that $v(f_0)\geq 0$. We then have that 
$$ \frac{c}{1-aX} + \frac{f(X)-f_0}{g(X)}\in \Oo_{K(X)}$$ 
Write $f(X)-f_0=X\cdot \tilde{f}(X)$. If $g_0=0$, then $g(X)=X\tilde{g}(X)$ and therefore  
$$\frac{c}{1-aX} + \frac{\tilde{f}(X)}{\tilde{g}(X)} \in \Oo_{K(X)}$$
Since $\text{deg}(\tilde{f})<\text{deg}(f)$, this contradicts our minimality assumption. Therefore $g_0\neq 0$. But then $v(f_0+cg_0) =vc<0$, which contradicts (1).
\qedhere $_{\textit{Claim 1}}$ \end{proof}
Next, we prove:\\
\textbf{Claim 2: }For each $m=0,...,n$, there exists $k_m\in \N$ such that $v(f_m)=v(c)+k_m \cdot v(a)$.
\begin{proof}
We proceed inductively. For $m=0$: Recall from (1) that $v(f_0+c g_0)\geq 0$. By Claim 1, we have $v(f_0)<0$. Therefore $g_0\neq 0$ and $v(f_0)=v(cg_0)=vc$, which settles the base case. For $m>0$: Recall that we have
$$(2)\ v(f_m+c g_m-af_{m-1})\geq 0$$ 
By our induction hypothesis, we have 
$$v (f_{m-1})=v(c)+k_{m-1}\cdot v(a)$$ 
for some $k_{m-1}\in \N$. Since $vc<\Z va$, this implies that $v(af_{m-1})<0$. If $g_m=0$, we get from (2) that
$$v(f_m)=v(af_{m-1})=v(c)+ (k_{m-1}+1)v(a)$$ 
and we take $k_m=k_{m-1}+1$. Suppose that $g_m\neq 0$. Then $v(c g_m-af_{m-1})=v(cg_m)=v(c)$. By (2), we get that $v(f_m)=v(c)$ and we take $k_m=0$. 
\qedhere $_{\textit{Claim 2}}$ \end{proof}
For $m=n$, we get that $v(f_n)=v(c)+k_n v(a)$, for some $k_n\in \N$. Therefore 
$$v(af_n)=v(c)+(k_n+1)v(a)<0$$ 
because $v(c)<\Z v(a)$. This contradicts (3).
\end{proof}

\bc \label{keylemma}
Let $(K,v,\text{res},\iota)$ be an e.c. model of $\text{VF}_{\text{res},\iota}$. Let $a\in  \mathfrak{m} - \{0\}$ and $c\in K$ be such that $v(c)<\Z v(a)$. Then there exists $\beta \in k^{\times}$ such that 
$$\text{res}(\frac{c}{1-a \beta })=0$$
\ec 
\begin{proof}
Immediate from Lemma \ref{trickylemma}.
\end{proof}

\subsection{Non-existence of a model companion of $\text{VF}_{\text{res},\iota}$}

\subsubsection{Generalities on model companions}
\begin{definition} [Definition 3.2.8 \cite{tentziegler}]
Let $T$ be a theory. A theory $T^*$ is a model companion of $T$ if the following conditions are satisfied:
\begin{enumerate}[label=(\roman*)]
\item Every model of $T$ embeds into a model of $T^*$.
\item Every model of $T^*$ embeds into a model of $T$.
\item $T^*$ is model-complete 
\end{enumerate}
\end{definition}

\begin{fact} [Theorem 3.2.9 \cite{tentziegler}]
A theory $T$ has, up to equivalence, at most one model companion $T^*$.
\end{fact}

\begin{definition}
\begin{enumerate}[label=(\roman*)]
\item Let $M,N$ be $L$-structures. Suppose that $N\models \phi \Rightarrow M\models \phi$ for any existential sentence $ \phi \in L(M)$. Then we say that $M$ is \textit{existentially closed} (or e.c.) in $N$ and write $M\preceq_{\exists} N$.
\item Let $T$ be a theory. A model $M\models T$ is said to be an existentially closed (or e.c.) model of $T$ if $M\preceq_{\exists} N$, for all $N\models T$.
\end{enumerate}
\end{definition}

\begin{fact} [Theorem 3.2.14 \cite{tentziegler}] \label{tentfact}
For any theory $T$, the following are equivalent:
\begin{enumerate}[label=(\roman*)]
\item $T$ has a model companion $T^*$.
\item The e.c. models of $T$ form an elementary class.
Moreover, if $T^*$ exists, then $T^*$ is the theory of e.c. models of $T$.
\end{enumerate}
\end{fact}

\subsubsection{Proof of Theorem \ref{modelcomp}}
\begin{Theor}
The theory $\text{VF}_{\text{res},\iota}$ does \textit{not} admit a model companion.
\end{Theor}  
\begin{proof}
Assume otherwise and let $T$ be the model companion of $\text{VF}_{\text{res},\iota}$. 
By Fact \ref{tentfact}, the theory $T$ is precisely the theory of e.c. models of $\text{VF}_{\text{res},\iota}$. Let $(K_0,v_0,\text{res}_0,\iota_0)\models T$ be $\aleph_1$-saturated. By Corollary \ref{lemmapowers} and $\aleph_1$-saturation, there is $a \in \mathfrak{m}_0-\{0\}$ such that $\text{res}_0(a^{-m})=0$ for $m>0$. Let 
$$(K,v,\text{res},\iota):=(K_0,v_0,\text{res}_0,\iota_0)^U$$ 
where $U$ is a non-principal ultrafilter on $\N$. By \L o\'s' Theorem, we will have that $(K_0,v_0,\text{res}_0,\iota_0)\preceq (K,v,\text{res},\iota)$. In particular, we get that $(K,v,\text{res},\iota)\models T$ and therefore $(K,v,\text{res},\iota)$ is an e.c. model of $\text{VF}_{\text{res},\iota}$. 

Set $a^*:=\text{ulim } a^{-n}$. Since $U$ is non-principal, we get that $va^*<\Z va$. By Corollary \ref{keylemma}, there is $\beta\in k-\{0\}$ such that $\text{res}(\frac{a^*}{1-a \beta })=0$. By \L o\'s' Theorem, we get that 
$$\{n\in \N: K_0\models \exists y \in k^{\times}( \text{res}_0(\frac{a^{-n}}{1-ay})=0)\}\in U$$ 
In particular, there exists $n\in \N$ and $\beta_0\in \iota(k_0)-\{0\}$ such that $\text{res}(\frac{a^{-n}}{1-a \beta_0})=0$. Note that $$\frac{a^{-n}}{1-a \beta_0} \equiv \beta_0^n+a^{-1}\beta_0^{n-1}+...+a^{-n} \mod \mathfrak{m}$$ 
Since $\text{res}:K\to k$ is $k$-linear and $\text{res}(\mathfrak{m})=\{0\}$ and $\text{res}(a^{-m})=0$ for all $m\in \N$, we compute that 
$$\text{res}(\frac{a^{-n}}{1-a \beta_0})=\text{res}(\beta_0^n+a^{-1}\beta_0^{n-1}+...+a^{-n})=\sum_{m=0}^n \beta_i^{m-i} \text{res}(a^{-m})=\beta_0^n $$
This forces $\beta_0=0$, which is a contradiction. It follows that $\text{VF}_{\text{res},\iota}$ does not admit a model companion.
\end{proof}

\section{Undecidability of $k(\!(t^{\Gamma})\!)$ with a total residue map}
Let $k(\!(t^{\Gamma})\!)$ be the Hahn field with residue field $k$ and value group $\Gamma$. Recall that an element $f \in k(\!(t^{\Gamma})\!)$ is of the form 
$$f=\sum_{q\in \Gamma} c_q t^q $$
where $\supp(f)=\{q\in \Gamma: c_q\neq 0\}$ is well-ordered. Throughout, we fix some element $1\in \Gamma^{>0}$, thereby identifying a copy of $\Z$ inside $\Gamma$. We write $t$ for $t^1$.
\subsection{Definability in $k(\!(t^{\Gamma})\!)$ in $L_{\text{res}}$}

\bl \label{resfieldef}
Let $k$ be any field and $\Gamma$ be any ordered abelian group. The subfield $k\subseteq k(\!(t^{\Gamma})\!)$ is $\emptyset$-definable in $L_{\text{res}}$. In particular, the lift $\iota:k\to k(\!(t^{\Gamma})\!)$ is $\emptyset$-definable in $L_{\text{res}}$.
\el 
\begin{proof}
We claim that 
$$k=\{x\in k(\!(t^{\Gamma})\!): \forall y (\text{res}(x\cdot y)=\text{res}(x)\cdot \text{res}(y) ) \}$$ 
The inclusion "$\subseteq$" is clear.
For "$\supseteq$", suppose that $x\notin k$ and write $x= \alpha+ x'$ with $\alpha \in k$, $\text{res}(x')=0$ and $x' \neq 0$. Note that $v(x')\neq 0$. For $q=v(x')$, we compute that
$$\text{res}(x\cdot t^{-q})=\alpha \cdot \text{res}(t^{-q})+\text{res}(x' \cdot t^{-q})=\text{res}(x'\cdot t^{-q})\neq 0 $$ 
On the other hand, we have  $\text{res}(x)\cdot \text{res}(t^{-q})=0$ and hence
$$\text{res}(x\cdot t^{-q}) \neq \text{res}(x)\cdot \text{res}(t^{-q})$$ 
Finally, for $\iota$ simply note that $\iota(\alpha)=a$ if and only  if $a\in \iota(k)$ and $\text{res}(a)=\alpha$.
\end{proof}

\begin{definition}
Given 
$$a=\sum_{q\in A} c_q t^q\in k(\!(t^{\Gamma})\!)$$ 
we define the polynomial $p_a(X)\in k[X]$ given by
$$p_a(X)=\sum_{n\in  \N} c_{-n} X^n$$ 
Note that $p_a(X)$ is indeed a polynomial because $\supp(a) \cap \Z^{\leq 0}$ is finite.
\end{definition}

\bl \label{polynomialpart}
For any $a \in k(\!(t^{\Gamma})\!)$ and $y\in k$, we have
$$\text{res}(\frac{a}{1-ty})=p_a(y)$$ 
\el 
\begin{proof}
Write $a=\sum_{q\in \Gamma} c_q t^q\in k(\!(t^{\Gamma})\!)$. For each $q_0 \in \Gamma$, note that 
$$\text{res}(t^{q_0} \cdot \sum_{q\in A} c_q t^q)=c_{-q_0} $$
We now have 
$$\frac{a}{1-ty} =(1+ty+...+t^ny^n+...) \cdot \sum_{q\in \Gamma} c_q t^q \equiv \sum_{i=1}^n t^iy^i \cdot \sum_{q\in \Gamma} c_q t^q \mod \mathfrak{m} $$
where $n$ is maximum such that $-n \in \text{supp}(a)$. By $k$-linearity of $\text{res}$ and since $\text{res}(\mathfrak{m})=\{0\}$, we get that
$$\text{res}(\frac{a}{1-ty})= \sum_{i=1}^n \text{res}(t^i \cdot  \sum_{q\in \Gamma} c_q t^q) y^i= \sum_{i=1}^n c_{-i}y^i =p_a(y)$$ 
as needed.
\end{proof}

\subsection{Infinite residue field}
Let $k$ be an \textit{infinite} field.
\subsubsection{Undecidability of $k(\!(t^{\Gamma})\!)$ in $L_{\text{res},t}$} \label{undecidabilitysect}
We isolate the following elementary fact from algebra to pinpoint exactly where our proof fails when $k$ is finite:
\bl \label{trivialemma}
Let $k$ be any infinite field and $f(X)\in k[X]$. Then $f(X)\equiv 0$ if and only if $f(k)=\{0\}$.
\el 
\begin{proof}
Any non-zero polynomial over any field has finitely many roots.
\end{proof}

\begin{fact} [Denef] \label{denfact}
Let $R=k[t]$, where $k$ is any field. Then Hilbert's tenth problem over $R$ with coefficients in $\Z[t]$ is unsolvable.
\end{fact}
\begin{proof}
See \cite{DenefPoly1} and \cite{DenefPoly2}.
\end{proof}

\begin{Theor} \label{mainthm}
Let $k$ be an infinite field and $\Gamma$ be any non-trivial ordered abelian group. Then $(k(\!(t^{\Gamma})\!),\text{res})$ is $\exists \forall$-undecidable in $L_{\text{res},t}$. The same is true for the Puiseux series field $k\{\!\{t\}\!\}$.
\end{Theor} 
\begin{proof}
Let $f_1(X_1,...,X_m,T),...,f_n(X_1,...,X_m,T)\in \Z[X_1,...,X_m,T]$, where $n,m\in \N$. The system
$$f_1(X_1,...,X_m,T)=....=f_n(X_1,...,X_m,T)=0$$
has a solution in $k[T]$ if and only if there exist $a_1,...,a_m\in k(\!(t^{\Gamma})\!)$ such that
$$f_1(p_{a_1}(t^{-1}) ,...,p_{a_m}(t^{-1}),t^{-1} )=....=f_n( p_{a_1}(t^{-1}) ,...,p_{a_m}(t^{-1}),t^{-1})=0$$
By Lemma \ref{trivialemma}, this is also equivalent to the existence of $a_1,...,a_m\in k(\!(t^{\Gamma})\!)$ such that
$$f_1(p_{a_1}(y),...,p_{a_m} (y),y)=....=f_n(p_{a_1} (y),...,p_{a_m} (y),y)=0$$ for all $y\in k$. This is expressible in $L_{\text{res},t}$ by Lemma \ref{polynomialpart} and Lemma \ref{resfieldef} and we conclude from Fact \ref{denfact}. The proof works verbatim for the Puiseux series field. 
\end{proof}
In \S 7.2.21 \cite{KartasThesis}, it is also shown that $k[t]\subseteq k(\!(t)\!)$ is definable in $L_{\text{res},t}$.
\subsubsection{A natural example from complex analysis}
In complex analysis, one defines the \textit{residue} of a complex function $f\in \C(\!(t)\!)$ at an isolated singularity $a\in \C$, denoted by $\text{Res}(f,a)$ or $\text{Res}_a(f)$. Numerically, if $f=\sum_{i=-n}^{\infty} c_i (t-a)^i$, we have that $\text{Res}(f,a)=c_{-1}$.

\bc \label{complexnumbers}
We have that $\C(\!(t)\!)$ is $\exists \forall$-undecidable in $L_{\text{Res}_0,t}$. The same is true for the Puiseux series field $\Cc\{\!\{t\}\!\}$.
\ec 
\begin{proof}
Note that $\text{Res}_0(a)=\text{res}(t\cdot a)$ and conclude from Theorem \ref{mainthm}.
\end{proof}
\subsection{Eliminating $t$} \label{elimint}
We now give a different proof of the undecidability of $(k(\!(t^{\Gamma})\!),+,\cdot,\text{res})$, where $k$ is an infinite field and $\Gamma$ is an ordered abelian group of rank $1$. The proof discussed here does not require a parameter for $t$ and even shows that $(\N,+,\cdot)$ is interpretable without parameters. 
\subsubsection{Interpreting the weak monadic second-order theory of $k$ }
Weak monadic second-order logic is the fragment of second-order logic where second-order quantification is restricted to quantification over \textit{finite} subsets. 
\bl \label{eliminatingtlemma}
Let $\Gamma$ be an ordered abelian group of rank $1$. Let $a\in k(\!(t^{\Gamma})\!)$ and $b\in \mathfrak{m}$. Then the set
$$S_{a,b}=\{\beta \in k: \text{res}(\frac{a}{1-b\beta})=0\}$$
is either finite or equal to $k$.
\el 
\begin{proof}
Since $\Gamma$ is of rank $1$, there exists $n\in \N$ such that $a b^n \in \mathfrak{m}$. We then have that
$$ \frac{a}{1-b\beta}\equiv a(1+b\beta+...+b^{n-1}\beta^{n-1}) \mod \mathfrak{m}$$
Since $\text{res}$ is $k$-linear and $\text{res}(\mathfrak{m})=\{0\}$, we get that
$$\text{res}(\frac{a}{1-b\beta})=0\iff p_{a,b}(\beta) =0$$
where 
$$p_{a,b}(X)=\sum_{i=0}^{n-1} \text{res}(ab^i)\cdot X^i \in k[X]$$ 
If $p_{a,b}(X)\equiv 0$, then $S_{a,b}=k$ and otherwise $S_{a,b}$ is finite.
\end{proof}

\bl \label{wmsointer}
Let $\Gamma$ be an ordered abelian group of rank $1$. Then the weak monadic second-order theory of $(k,+,\cdot)$ is $\emptyset$-interpretable in $k(\!(t^{\Gamma})\!)$ in $L_{\text{res}}$. The same is true for the Puiseux series field $k\{\!\{t\}\!\}$.
\el  
\begin{proof}
By Lemma \ref{resfieldef}, we have that the first-order theory $(k,+,\cdot)$ is $\emptyset$-interpretable in $k(\!(t^{\Gamma})\!)$ in $L_{\text{res}}$. By Lemma \ref{eliminatingtlemma}, we have a uniformly $\emptyset$-definable family 
$$\{S_{a,b}:a\in k(\!(t^{\Gamma})\!) \mbox{ and }b\in \mathfrak{m}\}$$ 
of finite subsets of $k$ together with $k$ itself. We claim that every finite $S\subseteq k$ arises as $S_{a,b}$, for some $a\in k(\!(t^{\Gamma})\!) \mbox{ and }b\in \mathfrak{m}$. Given a finite $S\subseteq k$, we set
$$a=\prod_{s\in S} (t^{-1}-s) \mbox{ and }b=t$$  
By Lemma  \ref{polynomialpart}, we get indeed that 
$$S_{a,b}=\{y:p_a(y)=0\}=S$$ 
We thus encode the weak monadic second-order theory of $(k,+,\cdot)$. The proof works verbatim for the Puiseux series field. 
\end{proof}
\bl  \label{secondordermonadic}
Let $k$ be an infinite field. Then the weak monadic second-order theory of $(k,+,\cdot)$ interprets $(\N,+,\cdot)$. In particular, it is undecidable.
\el 
\begin{proof}
See \S 1 \cite{Cherlin}.
\end{proof}

\begin{Theore} \label{interarithmetic}
Let $k$ be an infinite field and $\Gamma$ be an ordered abelian group of rank $1$. Then $(k(\!(t^{\Gamma})\!),\text{res})$ interprets $(\N,+,\cdot)$. In particular, $k(\!(t^{\Gamma})\!)$ is undecidable in $L_{\text{res}}$. The same is true for the Puiseux series field $k\{\!\{t\}\!\}$.
\end{Theore}
\begin{proof}
From Lemma \ref{wmsointer} and Lemma \ref{secondordermonadic}.
\end{proof}

\begin{rem}
In case $k$ is an infinite perfect field of positive characteristic, the undecidability of $(k(\!(t)\!),+,\cdot, \text{res})$ also follows from Lemma \ref{resfieldef} and \S 2 \cite{Cherlin} which shows that $k(\!(t)\!)$ is undecidable with a predicate for $k\subseteq k(\!(t)\!)$.
\end{rem}

\subsection*{Acknowledgements}
I wish to thank E. Hrushovski for suggesting the problem and for an instructive discussion on the elimination of the parameter. I also thank J. Koenigsmann for careful readings of earlier drafts.

\bibliographystyle{alpha}
\bibliography{references3}
\Addresses
\end{document}